\documentclass[11pt]{article}
\usepackage[utf8]{inputenc}
\usepackage{lmodern}
\usepackage{subfiles}
\usepackage{enumitem}
\setenumerate{topsep=6pt,ref={\normalfont(\roman*)},label={\normalfont(\roman*)}, itemsep=0pt} 
\usepackage[indent=15pt,skip=5pt]{parskip}

\usepackage{amsfonts}
\usepackage{amsthm}
\usepackage{amsmath}
\usepackage{amssymb}
\usepackage{amscd}
\usepackage{mathrsfs}
\usepackage{mathtools}
\usepackage{bbm}
\usepackage{esint}

\usepackage[margin=3cm]{geometry}
\usepackage{setspace}
\usepackage{indentfirst}
\usepackage{graphicx}
\usepackage{graphics}
\usepackage{lscape}
\usepackage{pgf,tikz}
\usepackage{tikz-cd}
\usepackage{color}
\usepackage{pict2e}
\usepackage{epic}
\usepackage{epstopdf}
\usepackage{titlesec, titlefoot}
\titleformat{\section}[block]{\Large\bfseries\filcenter}{\thesection}{1em}{}
\usepackage{commath}
\usepackage{float}
\usepackage{caption}
\usepackage{etoolbox}
\usepackage[affil-it]{authblk}
\usepackage{combelow}

\usepackage[hidelinks, bookmarksdepth=3,citecolor=blue,colorlinks]{hyperref}
\hypersetup{bookmarksopen=true} 
\usepackage{hypcap}

\graphicspath{{./Pictures/}}
\allowdisplaybreaks

\expandafter\def\expandafter\normalsize\expandafter{%
\normalsize
\setlength\abovedisplayskip{6pt}
\setlength\belowdisplayskip{6pt}
\setlength\abovedisplayshortskip{6pt}
\setlength\belowdisplayshortskip{6pt}
}

\theoremstyle{plain}

\renewcommand*\thesection{\arabic{section}}
\numberwithin{equation}{section} 

\newtheorem{theorem}{Theorem}[section]
\newtheorem{lemma}[theorem]{Lemma}
\newtheorem*{lemma*}{Lemma}
\newtheorem{proposition}[theorem]{Proposition}
\newtheorem{corollary}[theorem]{Corollary}

\theoremstyle{definition}
\newtheorem{definition}[theorem]{Definition}

\newtheorem{example}[theorem]{Example}
\newtheorem{question}[theorem]{Question}

\expandafter\let\expandafter\oldproof\csname\string\proof\endcsname
\let\oldendproof\endproof
\renewenvironment{proof}[1][\proofname]{%
\oldproof[\upshape \bfseries #1]%
}{\oldendproof}

\makeatletter
\def\@makechapterhead#1{%
\vspace*{50\p@}%
{\parindent \z@ \raggedright \normalfont
\interlinepenalty\@M
\Huge\bfseries  \thechapter.\quad #1\par\nobreak
\vskip 40\p@
}}
\makeatother

\renewcommand{\Re}{\operatorname{Re}}
\renewcommand{\Im}{\operatorname{Im}}

\newcommand{\eps}{\varepsilon}

\DeclareMathOperator{\diam}{diam}
\DeclareMathOperator \dist{dist}
\DeclareMathOperator{\Sym}{Sym}
\DeclareMathOperator{\cof}{cof}
\DeclareMathOperator{\tr}{tr}

\DeclareMathOperator{\SO}{SO}

\def \a{\alpha}
\def \Om{\Omega}
\def \R {\mathbb{R}}

\def \C{\mathbb{C}}

\def \N{\mathbb{N}}
\def \D{\textup{D}}

\def \e{\varepsilon}
\def \d{\,\textup{d}}

\def \p{\partial}
\def \mc{\mathcal}

\def \w{\rightharpoonup}

\def \tp{\textup}
\def \Id{\textup{Id}}

\def \loc{\textup{loc}}

\begin{document}

	\title{\textbf{Unique continuation for differential inclusions}}
	
	\author[1]{{\Large Guido De Philippis}}
	\author[2]{{\Large Andr\'e Guerra}}
	\author[3]{{\Large Riccardo Tione}}
	
	\affil[1]{\small Courant Institute of Mathematical Sciences, New York University,
	\protect \\ 251 Mercer St., New York, NY 10012, USA
	\protect\\
	{\tt{guido@cims.nyu.edu}} \vspace{1em} \ }
	
	\affil[2]{\small Institute for Theoretical Studies,  ETH Zürich,  Clausiusstrasse 47, 8006 Zürich,  Switzerland
	\protect\\
	{\tt{andre.guerra@eth-its.ethz.ch}} \vspace{1em} \ }
	
	\affil[3]{\small Max Planck Institute for Mathematics in the Sciences, Inselstrasse 22, 04103 Leipzig, Germany
	\protect\\
	{\tt{riccardo.tione@mis.mpg.de}}  }
	
	\date{}
	
	\maketitle

	\begin{center}
	\vspace{-0.4cm}
	\textit{\large Dedicated to Luigi Ambrosio, in the occasion of his 60th anniversary. \vspace{1em} }
	\end{center}
	
	\begin{abstract}
We consider the following question arising in the theory of differential inclusions: given an \emph{elliptic} set $\Gamma$ and a Sobolev map $u$ whose gradient lies in the quasiconformal envelope of $\Gamma$ and touches $\Gamma$ on a set of positive measure,  must $u$ be affine? We answer this question positively for a suitable notion of ellipticity, which for instance encompasses the case where $\Gamma \subset \R^{2\times 2}$ is an elliptic, smooth, closed curve.  More precisely, we prove that the distance of $\D u$ to $\Gamma$ satisfies the strong unique continuation property.
As a by-product, we obtain new results for nonlinear Beltrami equations and recover known results for the reduced Beltrami equation and the Monge--Amp\`ere equation: concerning the latter,  we obtain a new proof of the $W^{2,1+\e}$-regularity for two-dimensional solutions.
	\end{abstract}
	
	
\section{Introduction}
	
Let $\Omega\subset \R^n$ be a connected open set and consider a subset $\Gamma\subset \R^{n\times n}$. In this paper we study solutions $u\in W^{1,n}_\loc(\Omega,\R^n)$ to the differential inclusion associated with the \emph{K-quasiconformal envelope} of $\Gamma$, 
\begin{equation*}
\label{eq:inc}
\D u(x) \in \mc{E}_\Gamma  \quad \text{for a.e.\ } x\in \Omega.
\end{equation*}
More precisely, and following \cite{Faraco2008,Kirchheim2008}, for a fixed $K\in [1,\infty)$ we consider the set
\begin{equation}\label{eq:Kqc}
\mc E_\Gamma \equiv \{X\in \R^{n\times n}: |A-X|^n\leq K \det (A-X) \tp{ for all } A\in \Gamma\},
\end{equation}
where $|\cdot|$ denotes the operator norm.  For brevity we will omit $K$ from the definition of $\mc{E}_\Gamma$. 

\subsection{Main results}

The question we answer in this paper is the following:
\begin{question}\label{Q}
Let $\Gamma \subset \R^{n\times n}$ be an \emph{elliptic} set. Suppose $u\in W^{1,n}_\loc(\Omega,\R^n)$ satisfies $\D u\in \mc E_\Gamma$ a.e.\ in $\Omega$. If $|\{x \in \Omega: \D u \in \Gamma\}|> 0$, is then $u$ affine?
\end{question}

The heuristic reason  why one would expect this question to have a positive answer comes from the following observation. By definition of $\mc{E}_\Gamma$, a map $u\in W^{1,n}_\loc(\Omega,\R^n)$ satisfies $\D u \in \mc E_\Gamma$ a.e.\  if and only if the map $u-A$ is $K$-quasiregular for all $A\in \Gamma$, see already Definition \ref{eq:qr}. Now, $K$-quasiregular maps enjoy strong rigidity properties: it is a deep analytic fact, due to Reshetnyak, that if $f$ is quasiregular and if $\det(\D f)=0$ on a set of positive measure,  then $f$ is constant \cite{Reshetnyak1989}.  In fact, this property of quasiregular mappings is the most basic case covered by Question \ref{Q}, as it corresponds to taking $\Gamma = \{0\}$.  It follows from this discussion that a non-affine map $u$ having the properties expressed in Question \ref{Q} must be such that
\[
|\{x \in \Omega: \D u \in \Gamma\}|> 0, \quad \text{ but }\quad |\{x \in \Omega: \D u = A\}|= 0,\; \forall A \in \Gamma,
\]
i.e.\ it must be quite pathological.  We note, however, that the relatively low regularity assumed on $u$ is not an essential difficulty in Question \ref{Q}: indeed,  even if $u$ is assumed to be smooth,  it seems to be a non-trivial task to rule out the pathological behavior just described.

It is clear that, in order to answer Question \ref{Q}, one needs to specify a notion of \emph{ellipticity}.  Generally speaking, one could call a set $\Gamma$ \emph{elliptic} if (i) Lipschitz maps satisfying $\D u \in \Gamma$ a.e.\ are actually $C^{1,\alpha}$, and if (ii) weakly convergent sequences of maps whose gradients approach $\Gamma$ actually converge strongly. This notion  of ellipticity is,  however,  too weak for our purposes.  Instead, we will study sets $\Gamma$ having the following stronger property:

\begin{definition}\label{def:main}
We say that a set $\Gamma\subset \R^{n\times n}$
satisfies a \textit{rigidity estimate} if there is a constant $C_\Gamma>0$ such that, for all balls $B\subset \R^n$ and all $v\in W^{1,n}(B,\R^n)$, we have
	\begin{equation}
	\label{eq:weakrig}
	\inf_{\substack{A \in \Gamma}}
 \int_{\frac 1 2 B} |\D v-A|^n \d x
	\leq C_\Gamma \int_{B} \dist(\D v, \Gamma)^n \d x.
	\end{equation}		
\end{definition}

We will discuss this  definition in more detail  below, but for now let us observe that any set $\Gamma$ which satisfies \eqref{eq:weakrig} is very rigid:  the only maps $\varphi \in W_{\loc}^{1,n}(\Omega,\R^n)$ fulfilling $\D\varphi \in \Gamma$ a.e.\ are affine.  The prototype of a set $\Gamma$ satisfying this condition is 
$$\Gamma = \SO(2) \equiv \{A \in \R^{2\times 2}: A^TA = \Id, \det(A) > 0\},$$
as first shown in \cite{Friesecke2002}.  This example is extremely important as it is related to the Burkholder function and quasiregular maps on one hand \cite{Astala2022,Baernstein1997a,Sverak1990} and to the Monge--Amp\`ere equation on the other \cite{Ambrosio2011}, the latter connection playing a key role in the present paper. We refer the reader to \cite{Hirsch2023,Lamy2019,Muller1999a, Sverak1993,Tione2021,Zhang1997} for more results concerning elliptic and non-elliptic differential inclusions.

Our first main theorem provides a positive answer to Question \ref{Q}, by asserting a unique continuation principle for the distance function:	
	
	\begin{theorem}\label{thm:main}
	Let $\Gamma\subset \R^{n\times n}$ satisfy a   rigidity estimate. Let $u \in W^{1,n}_\loc(\Omega,\R^n)$ satisfy 
\begin{equation}
\label{eq:incEgamma}
\D u\in \mc E_\Gamma\quad \text{a.e.\ in } \Omega.
\end{equation}
	Then either $\dist(\D u, \Gamma)>0$ a.e.\ in $\Omega$,  or $\dist(\D u, \Gamma)=0$ a.e.\ in $\Omega$, in which case  $u$ is affine. 
	\end{theorem}
		
In fact, we can show the following quantitative version of Theorem \ref{thm:main}:
	
	\begin{theorem}\label{thm:qttive}
	Let $\Gamma\subset \R^{n\times n}$ satisfy a   rigidity estimate,  let $u \in W^{1,n}(\Omega,\R^n)$ satisfy \eqref{eq:incEgamma} and take $\Omega' \Subset \Omega$ open.  Suppose that there is $M > 0$  such that
\begin{equation}
\label{eq:bdddistance}
0 < M^{-1} \leq \int_{\Omega'} \dist(\D u, \Gamma)^n \d x \quad \text{ and } \quad
\int_\Omega |\D u|^n \d x \leq M.
\end{equation}	
Then, there exists $\e=\e(n,\Gamma,M,K,\Omega')>0$ and $C = C(n,\Gamma,M,K,\Omega')>0$ such that 
\begin{equation}\label{eq:unif}
\int_{\Omega'} \dist(\D u, \Gamma)^{-\e} \d x \le C.
\end{equation}
	\end{theorem}		

We now discuss applications of our results in specific examples, which in fact served as motivation for thinking about Question \ref{Q} in the first place.

\subsection{Elliptic curves and nonlinear Beltrami equations}\label{sec:curves}

The following is the main example we will discuss in this subsection:

\begin{example}[Elliptic curves]
We say that $\Gamma\subset \R^{2\times 2}$ is a \textit{$K$-elliptic curve} if it is the image of a closed, smooth curve $\gamma\colon [0,1] \to \R^{2\times 2}$ without self-intersections and satisfying
\begin{equation}\label{def:ell}
|\gamma(t) - \gamma(s)|^2 \le K\det(\gamma(t)-\gamma(s)), \quad \forall s,t \in [0,1].
\end{equation}
Elliptic curves were first introduced in \cite{Szekelyhidi2005} in the study of rank-one convex hulls of compact sets $K \subset \R^{2\times 2}$.  The \emph{separation properties} they provide have then led to striking achievements in the following years  \cite{Faraco2008,Kirchheim2008}.  More recently,  it was shown in  \cite{Lamy2023} that elliptic curves satisfy a   rigidity estimate, and thus our theorems apply in this setting. 
\end{example}
 
%
%
For instance, it is easy to see that  $\Gamma = \SO(2)$ is a 1-elliptic curve.  More generally,  any smooth curve $\Gamma$ contained in the conformal plane $\tp{CO}(2)\equiv \tp{span}(\SO(2))$ is elliptic. In this case,  a map $u$ solves \eqref{eq:incEgamma} if and only if it solves the \textit{nonlinear Beltrami equation}
\begin{equation}
\label{eq:belt}
\p_{\bar z} u = \mu \dist(\p_z u,\Gamma), \qquad \|\mu\|_{\infty}\leq \frac{K-1}{K+1} ,
\end{equation}
for some Beltrami coefficient $\mu$. In \eqref{eq:belt} we use the usual Wirtinger derivatives, which allow us to identify $\D u\in \R^{2\times 2}$ with $(\p_z u, \p_{\bar z} u)\in \C^2$, see Section \ref{sec:belt} for further details. 

Nonlinear Beltrami equations, with general nonlinearities,  have been studied extensively in recent years  \cite{Astala,Astala2019,Astala2017,Astala2012a,Astala2001}.  For equations with the structure in \eqref{eq:belt},  Theorem \ref{thm:main} yields the following novel conclusion: 

\begin{corollary}[Nonlinear Beltrami equations]
\label{cor:nonlinbeltr}
Let $\Gamma \subset \tp{CO}(2)$ be a $K$-elliptic curve and let $u \in  W_{\loc}^{1,2}(\Omega,\R^2)$ be a non-affine solution to \eqref{eq:belt}. Then $\dist(\p_z u,\Gamma)\neq 0$ a.e.\ in $\Omega$. 
\end{corollary}

For general nonlinear Beltrami equations, the analogue of Corollary \ref{cor:nonlinbeltr} is known only when $K<2$ \cite{Astala}. We can specialize Corollary \ref{cor:nonlinbeltr} to the case $\Gamma=\tp{span}(\Id)$; strictly speaking, this is not a closed curve, but the result holds true nonetheless. We then obtain:

\begin{corollary}[Reduced Beltrami equation]
\label{cor:jaask}
Any solution $u \in  W_{\loc}^{1,2}(\Omega,\R^2)$ to 
\begin{equation}
\label{eq:redbelt}
\partial_{\bar z}u = \mu \Im(\partial_zu ), \qquad \|\mu\|_\infty\leq \frac{K-1}{K+1},
\end{equation}
is either affine or satisfies $\Im(\partial_z u) \neq 0$ a.e.\ in $\Omega$.
\end{corollary}

Equation \eqref{eq:redbelt} is known as the \textit{reduced Beltrami equation}, and one can reduce general linear elliptic systems in the plane  to it \cite[\S 6]{Astala2009}. Corollary \ref{cor:jaask} was first established  in \cite{Iwaniec2006} in the case where $u$ is a global homeomorphism, and under the additional assumption that $\|\mu\|_{L^\infty} < 1/2$. Next, in \cite{Alessandrini2009a},  this last condition on $\|\mu\|_{L^\infty}$ was removed, see also \cite{Astala2009a}.  The general case of Corollary \ref{cor:jaask}, without any homeomorphicity assumptions,  was finally obtained in \cite{Jaaskelainen2013}. 

We note that although Corollary \ref{cor:jaask} is known from \cite{Jaaskelainen2013} the proof we present here is much simpler.  In \cite{Jaaskelainen2013}, the author establishes a reverse H\"older inequality with \textit{increasing supports} for $|\Im(\partial_z u)|$, which implies that zeros of $|\Im(\partial_z u)|$ have infinite order; the author then performs a delicate analysis to rule out this possibility.  In the proof of Theorem \ref{thm:main} we establish directly a reverse H\"older inequality with the same support, from which the conclusion follows.

\subsection{SO(2) and the Monge--Amp\`ere equation}

In Section \ref{sec:curves} we discussed consequences of Theorem \ref{thm:main} for general elliptic curves. We now specialize to the case $\Gamma =\SO(2)$ and draw a connection to the Monge--Amp\`ere equation.  
For a general introduction to this equation, we refer the reader to  \cite{Figalli2017}.

If $O\subset \R^n$ denotes a convex set, we consider convex functions $\varphi \colon O \to \R$ such that
\begin{equation}\label{eq:MA}
\begin{cases}
\lambda^{-1} \d x \le \mu_\varphi \le \lambda \d  x,\\
\varphi = 0 \text{ on }\partial O,
\end{cases}
\end{equation}
where $\mu_\varphi$ denotes the \textit{Monge--Amp\`ere measure} of $\varphi$ and $\lambda>1$.   We recall that, whenever $\varphi\in W^{2,n}$, we have $\mu_\varphi=\det \D^2 \varphi.$ There is a unique convex solution to \eqref{eq:MA} and much attention has been given to understanding whether this solution  belongs to $W^{2,p}_{\loc}(O)$ for some $p \ge 1$.  A first perturbative result was obtained in \cite{Caffarelli1990}, see also \cite{Wang1995},  and then in \cite{DePhilippis2013} it was shown that
\begin{equation}\label{p:p1}
\varphi \in W^{2,1}_{\loc}(O),
\end{equation}
with a uniform modulus of equi-integrability of the second derivatives.  This result was then strengthened in \cite{DePhilippis2013a,Schmidt2013}, where it was shown that
\begin{equation}\label{p:p2}
\varphi\in W^{2,1+ \varepsilon}_{\loc}(O), \text{ for some } \varepsilon > 0 \text{ independent of $\varphi$}.
\end{equation}
We also refer the reader to \cite{Savin2020} for a global version of \eqref{p:p2} in general domains.

We will focus here on the case $n = 2$. A crucial idea in our strategy, which was introduced in \cite{Ambrosio2011},  is to apply Minty's correspondence between monotone maps and 1-Lipschitz maps. Precisely, if  $\varphi$ is the convex solution of \eqref{eq:MA} when $n=2$ then we consider the maps
$$\Phi_1(w) \equiv  \frac{\D\varphi(w) - w}{\sqrt{2}} \quad \text{ and }\quad  \Phi_2(w) \equiv \frac{\D\varphi(w) + w}{\sqrt{2}}.$$
Since $\varphi$ is convex, $\Phi_2$ is a homeomorphism.  There are now two main points,  both essentially due to  \cite{Ambrosio2011}. The first one, which we precise in Lemma \ref{lemma:qcenvelope}, is that if we set $v \equiv \overline{\Phi_1}\circ \Phi_2^{-1}$, where the bar denotes complex conjugation,  then 
$$ \D v \in\mc E_{\SO(2)}\cap \{A\in \R^{2\times 2}:|A|\leq 1\} \quad \text{ a.e.\ in } \Omega \equiv \Phi_2^{-1}(O);$$
here, the corresponding constant $K$ from \eqref{eq:Kqc} is precisely $K=\lambda$.  
The second point is that
$$\eqref{p:p1} \text{ when } n = 2 \quad \iff \quad  |\{\D v \in \SO(2)\}|=0. $$
It is easy to see that we cannot have $\dist(\D v, \SO(2))=0$ a.e.\ in $\Omega$, and thus Theorem \ref{thm:main} already gives a new proof of \eqref{p:p1} when $n=2$. In fact, as we show in Lemma \ref{lem:linked}, one has the stronger relationship
$$|\D^2 \varphi(x)|\sim_\lambda \frac{1}{\dist(\D v(\Phi_2(x)), \SO(2))}\quad \text{for a.e.\ } x \text{ in } \Omega,$$
and so the higher integrability of $\D^2 \varphi$ is linked precisely to the decay of the distance of $\D v$ to $\SO(2)$.
Thus, applying Theorem \ref{thm:qttive}, we in fact recover \eqref{p:p2} when $n=2$:

\begin{corollary}\label{cor:MA}
Let  $\varphi\colon O  \to \R$ be the unique convex solution to \eqref{eq:MA}.
Then $\varphi \in W^{2,1 + \varepsilon}_{\loc}(O)$ for some $\varepsilon = \varepsilon(\lambda)>0$.
\end{corollary}

\subsection{Non-examples}

We conclude this introduction by briefly mentioning further examples of sets to which either our theorems do not apply, or where they do apply but do not provide any  information. 

\begin{example}[Isometries and Conformal Maps]
If $\Gamma=\tp{SO}(n)$, then Liouville's Theorem asserts that rotations are the only exact solutions $v$ to $\D v \in \Gamma$ a.e.\ in $\Omega$. Moreover, by \cite{Friesecke2002}, $\tp{SO}(n)$ satisfies a rigidity estimate (strictly speaking,  in \cite{Friesecke2002} the estimate is only stated in $L^2$, but see also \cite[\S 2.4]{Conti2006} for the $L^p$-estimate). 
However, if $n \ge 3$,  then
$$\SO(n) \cap \mathcal{E}_{\SO(n)} = \emptyset,$$  since for every $X\in \SO(n)$ we can find some $Y \in \SO(n), Y \neq X,$ such that $\det(X-Y) = 0.$
Therefore,  in this case, Question \ref{Q} becomes meaningless. The same of course holds for any subset $\Gamma$ of the conformal matrices $\tp{CO}(n) \equiv \{A: A = tR, t\ge 0, R \in \SO(n)\}$ invariant under the natural $\SO(n)$-action,  such as 
\[
\Gamma= [m,M]\tp{SO}(n) \quad \text{ for } 0<m<M<\infty.
\]
For such sets,   rigidity estimates were shown in \cite{Faraco2005}.
\end{example} 

\begin{example}[General elliptic linear spaces]
If $\Gamma\subset\R^{n\times n}$ is a linear subspace without rank-one matrices, the rigidity estimate \eqref{eq:weakrig} may not hold, as solutions to $\D v \in \Gamma$ do not need to be affine. For instance, with 
$\Gamma = \tp{CO}(2)$,  the differential inclusion $\D v\in \Gamma$ simply says that $v$ is a holomorphic function. It is also interesting to notice that $\mathcal{E}_{\tp{CO}(2)} = \tp{CO}(2)$, and hence our results would not give any new information anyway. 
\end{example}


%
%

\subsection*{Outline} 
The paper is organized as follows. In Section \ref{sec:PAN}, we will introduce the notation and the general tools we will use in the paper, more precisely  quasiregular mappings and Muckenhoupt weights. In Sections \ref{sec:qualita} and in Section \ref{sec:quantita} we will show Theorems \ref{thm:main} and \ref{thm:qttive} respectively. 
In Section \ref{sec:belt} we will see how Corollary \ref{cor:jaask} follows from Theorem \ref{thm:main}.  In Section \ref{sec:MA} we review some useful results on the Monge--Amp\`ere equation, and show how Corollary \ref{cor:MA} follows from Theorem \ref{thm:qttive}.  Finally,  we included an appendix which contains the proof of a technical result concerning Muckenhoupt weights. 

\subsection*{Acknowledgements}
GDP has been partially supported by the NSF grant DMS 2055686 and by the Simons Foundation.  AG acknowledges the support of Dr. Max R\"ossler, the Walter Haefner Foundation and the ETH Z\"urich Foundation. AG and RT thank Daniel Faraco for interesting discussions concerning nonlinear Beltrami equations, and they thank the Hausdorff Institute for Mathematics (HIM) in Bonn,
funded by the Deutsche Forschungsgemeinschaft (DFG, German Research Foundation) under Germany's Excellence Strategy – EXC-2047/1 – 390685813,  and the organizers of the Trimester Program ``Mathematics for Complex Materials'' (03/01/2023-14/04/2023, HIM, Bonn) for their hospitality during the period where some of this work was conducted.

\section{Preliminaries and Notation}\label{sec:PAN}

 Throughout the paper,  $(a,b)$ denotes the standard scalar product between vectors $a,b \in \R^n$ and $\langle A,B\rangle$ the Hilbert-Schmidt scalar product between matrices $A, B\in \R^{n\times m}$. Similarly, the Euclidean norm of vectors $a \in \R^n$ is denoted by $|a|$ and the operator norm of matrices $A \in \R^{n\times m}$ is denoted by $|A|$. If $A \in \R^{n\times n}$, then $\det(A),\tr(A)$ and $A^T$ denote its determinant, trace and transpose respectively.  $\Sym(n)$ denotes the space of symmetric matrices in $\R^{n\times n}$ and, for $A\in  \Sym(n)$, we let $\cof(A)$ be the matrix defined as
\[
\cof(A)_{ij} = (-1)^{i + j}\det(M_{ij}(A))= (-1)^{i + j}\det(M_{ji}(A)),
\]
where $M_{ij}(A)$ denotes the $(n-1)\times(n-1)$ submatrix of $A$ obtained by eliminating from $A$ the $i$-th row and the $j$-th column. In particular, $\cof(A)$ satisfies
\[
A\cof(A) = \cof(A)A = \det(A)\Id.
\]

For a (Lebesgue) measurable set $E$, $|E|$ denotes its Lebesgue measure. For any set $D \subset \R^n$, $\overline{D}$ denotes its closure and $\partial D$ its topological boundary. $B_r(x)$ denotes the ball of $\R^n$ centered at $x$ with radius $r > 0$. If $B = B_r(x)$ for some $r,x$, then $\lambda B \equiv B_{\lambda r}(x)$ for any $\lambda > 0$.  We will always use $\Omega\subset \R^n$ to denote an open, nonempty,  connected set. Furthermore, for another open set $\Omega'$, $\Omega' \Subset \Omega$ means that $\overline{\Omega'} \subset \Omega$.

Given a set $E$ with $0 < |E|< +\infty$ and $f \in L^1_{\loc}(\R^n)$, we let
\[
\fint_{E}f\d x \equiv \frac{1}{|E|}\int_{E}f\d x
\]
be the average of $f$ over $E$.

Given positive functions $h$ and $g$, we will sometimes use the notation $h \lesssim g$ to say that there exists $C > 0$ such that $h \le Cg$. Analogously, $h \sim g$ means $h \lesssim g \lesssim h$. If we want to make the dependence of the constants explicit on some parameters $a,b$ etc., we will write $\lesssim_{a,b}$ and $\sim_{a,b}$.

\subsection{Quasiregular mappings}

We recall that a map $f\in W^{1,n}_\loc(\Omega,\R^n)$ is said to be \textit{$K$-quasiregular} if it satisfies the distortion inequality
\begin{equation}\label{eq:qr}
|\D f(x)|^n \leq K \det \D f(x) \quad \text{for a.e.\ } x\in \Omega.
\end{equation}
A quasiregular homeomorphism is said to be \textit{quasiconformal}.
Quasiregular maps are a far-reaching generalization of holomorphic functions, and we refer the reader to the monographs \cite{Iwaniec2001,Reshetnyak1989,Rickman1993} for a wealth of information on the topic. We will now recall the main definitions and results from this theory that we will need in the proofs of Theorems \ref{thm:main}-\ref{thm:qttive}.

The following deep topological result is essentially due to Reshetnyak, see e.g.\ \cite{Reshetnyak1989} or \cite[Theorem 16.12.1]{Iwaniec2001}:  

\begin{theorem}\label{thm:reshet}
Any $K$-quasiregular map $f \in W^{1,n}_{\loc}(\Omega,\R^n)$ is continuous, with modulus of continuity over $
\Omega'\Subset \Omega$ depending only on $K$ and on $\|\D f\|_{L^n(\Omega',\R^n)}$.  Moreover, if $f$ is non-constant then it is open and discrete.
\end{theorem}

Given a continuous map $f\colon \Omega\to \R^n$ and a set $\Omega'\Subset \Omega$, we write
\begin{equation}
\label{eq:defmult}
N(f,\Omega',y)\equiv \#(f^{-1}(y) \cap \Omega'),\qquad
N(f,\Omega')\equiv \sup_{y\in \R^n} N(f,\Omega',y).
\end{equation}
For a discrete and open continuous map (in particular, for a quasiregular map), the latter quantity is locally bounded \cite[I.\ Proposition 4.10(3)]{Rickman1993}:
\begin{lemma}[Bounded multiplicity]\label{lemma:multbounds}
Let $f\in C^0(\Omega, \R^n)$ be a discrete and open mapping. For any $\Omega'\Subset\Omega$ we have $N(f,\Omega')<\infty.$
\end{lemma}

We now move towards more analytic results. 
The next lemma, which is due to Martio \cite{Martio1974},  gives a fundamental analytic property of the derivatives of quasiregular maps. This result was extended to mappings of finite distortion in \cite{Hencl2007}.  In the quasiconformal case the lemma is a classical result of Gehring \cite{Gehring1973a}.

\begin{lemma}[Reverse H\"older inequality]\label{lemma:revholder}
Let $f\in W^{1,n}_\loc(\Omega, \R^n)$ be a $K$-quasiregular map. For all balls with $2B\subset \Omega$, we have
$$\left(\fint_{B} |\D f|^n \right)^{\frac 1 n} \d x\leq C(n,K,N(f,\Omega)) \fint_B |\D f| \d x.$$
Here, $C(n,K,N(f,\Omega)) = +\infty$ if $N(f,\Omega) = +\infty$.
\end{lemma}

Roughly speaking, Lemma \ref{lemma:revholder} shows that, whenever $f$ is quasiregular, $|\D f|^n$ is a $A_\infty$ weight \cite[Section 9.3]{Grafakos2009}.  More precisely,  the lemma asserts that $|\D f|^n \in A_{\infty,\loc}$, a space that we will introduce in Section \ref{sec:muc} below.  It is well-known that such weights are doubling:

\begin{lemma}[Doubling property]\label{lemma:doubling}
Let $f\in W^{1,n}_\loc(\Omega, \R^n)$ be a $K$-quasiregular map. For all balls $B$ with $2B\subset \Omega$, we have
$$\int_{B} |\D f|^n \d x \leq C(n,K,N(f,\Omega)) \int_{\frac 1 2 B} |\D f|^n \d x.$$
Here, $C(n,K,N(f,\Omega)) = +\infty$ if $N(f,\Omega) = +\infty$.
\end{lemma}

The reader should compare this lemma with  \cite[Lemma 5.2]{Hencl2007} (using also Lemma \ref{lemma:multbounds}). 


\subsection{Muckenhoupt weights}\label{sec:muc}

Besides quasiregular maps, for the proof of Theorem \ref{thm:qttive} we will also require a basic result concerning Muckenhoupt weights. We start by giving the definitions of $A_{p,\loc}(\Omega)$ and $A_{\infty,\loc}(\Omega)$ weights, which are based on the definition of the classical $A_p(\R^n),A_\infty(\R^n)$ weights. 
We set, for any $t \ge 1$,
\[
\mc{B}_{t}(\Omega) \equiv \{B: \text{$B$ is a ball with }tB \subset\Omega\}.
\]
We then introduce the following local definition:

\begin{definition}\label{def:locmuc}
Let $p > 1$ and $w \in L^1_{\loc}(\Omega)$. We say that $w \in A_{\infty,\loc}(\Omega)$ if it satisfies a reverse H\"older inequality: there exist $t,\eps, C > 0$ such that, for all $B \in \mc{B}_{t}(\Omega)$,
\begin{equation}\label{eq:ainf}
\left(\fint_{B}w^{1 + \eps}(x)dx\right)^{\frac{1}{1 + \eps}} \le C\fint_{B}w(x)dx.
\end{equation}
Moreover, we say that $w \in A_{p,\loc}(\Omega)$ if $w > 0$ a.e.\ on $\Omega$ and, for some $t>0$, 
\[
\sup_{B \in \mc{B}_t(\Omega)}\left(\fint_{B}w(x)dx\right) \left(\fint_{B}w^{-(p-1)}(x)dx\right)^{\frac{1}{p-1}}< +\infty.
\]
\end{definition} 

In the sequel, we will use the following general result concerning the relation between $A_{\infty,\loc}(\Omega)$ and $A_{p,\loc}(\Omega)$ weights:

\begin{theorem}[$A_p$-weights]\label{thm:Ap}
Let $w\in A_{\infty,\loc}(\Omega)$ and let $\eps,C> 0$ be the constants for which \eqref{eq:ainf} holds for balls in $\mc{B}_{t}(\Omega)$. Then either $w = 0$, or $w > 0$ a.e.\ in $\Omega$ and $w \in A_{p,\loc}(\Omega)$ for some $p<\infty$. In the latter case, there exist $c>0$ and $p>1$ depending only on $n,$ $\eps$ and $C$ such that
\begin{equation}
\fint_{B} w^{1-p}dx \le c \left(\fint_{B}w(x)dx\right)^{1-p},
\label{eq:revholderprelim}
\end{equation}
for all balls $B \in \mc{B}_{16t}(\Omega)$.
\end{theorem}

Theorem \ref{thm:Ap} is well-known but,  since we could not find a precise reference for it in our local setting,  for the sake of completeness we provide a proof in Appendix \ref{app:a}. However,  the proof is almost identical to the standard one in the usual global context of $A_p$ and $A_\infty$ weights, as in \cite[Theorem 9.3.3]{Grafakos2009}.


\section{A unique continuation principle for the distance function}\label{sec:qualita}

In this section we prove Theorem \ref{thm:main}. The main point of the proof is to show that, for a solution $u\in W^{1,n}_\tp{loc}(\Omega,\R^n)$ of \eqref{eq:incEgamma},  the distance function $\dist(\D u, \Gamma)$ satisfies a reverse H\"older inequality with non-increasing supports.

\subsection{Proof of Theorem \ref{thm:main}}

Let us fix a set $\Omega'\Subset \Omega$ and some small $\delta>0$. We claim that there is a positive constant $C=C(u,n,\Gamma, K,\Omega')$ such that 
\begin{equation}
\label{eq:revholderdist}
\left(\fint_B \tp{dist}(\D u,\Gamma)^{n+\delta} \d x\right)^{\frac{1}{n+\delta}}\leq C \left(\fint_B \tp{dist}(\D u,\Gamma)^n \d x\right)^{\frac 1 n}
\end{equation}
for all balls  $B$ with $2B\subset \Omega'$. Once this is shown, the conclusion follows by applying the principle of unique continuation for functions satisfying reverse H\"older inequalities, see \cite[Lemma 14.5.1]{Iwaniec2001},  or instead simply Theorem \ref{thm:Ap} above, which asserts that any such function which vanishes on a set of positive measure must vanish identically.

In order to prove \eqref{eq:revholderdist}, let $A \in \Gamma$ be any matrix, and let $\varphi(x) \equiv Ax$. We want to apply Lemma \ref{lemma:revholder} to the $K$-quasiregular map $f\equiv u-\varphi$, which we may assume to be non-constant, as otherwise there is nothing to prove. Note that, by Gehring's Lemma \cite[Lemma 3]{Gehring1973a}, Lemma \ref{lemma:revholder} yields the improved estimate
\begin{equation}\label{eq:rev}
\left(\fint_{B} |\D f|^{n+\delta}\d x \right)^{\frac {1}{n+\delta}} \leq C(n,K,N(f,\Omega')) \fint_B |\D f| \d x,
\end{equation}
for some small $\delta=\delta(n,K,N(f,\Omega'))>0$. Using this estimate, we have
\begin{align*}
\fint_B \dist(\D u,\Gamma)^{n+\delta} \d x & \leq \fint_B |\D (u- \varphi)|^{n+\delta} \d x \\
& \overset{\eqref{eq:rev}}{\leq} C(n,K,N(u-\varphi,\Omega')) \left(\fint_B |\D (u- \varphi)| \d x \right)^{n+\delta}\\
& \leq C(n,K,N(u-\varphi,\Omega')) \left(\fint_B |\D (u- \varphi)|^n \d x \right)^{\frac{n+\delta}{n}},
\end{align*}
where the last inequality is just H\"older's inequality. Thus, applying Lemma \ref{lemma:doubling} to the quasiregular map $f$, we arrive at
\begin{equation}\label{eq:prefinale}
\left(\fint_B \dist(\D u,\Gamma)^{n+\delta} \d x\right)^{\frac{1}{n+\delta}} \leq C(n,K,N(u-\varphi,\Omega')) 
\left(\fint_{\frac 1 2 B} |\D (u- \varphi)|^n \d x \right)^{\frac{1}{n}}.
\end{equation}
 Since we assume that $\Gamma$ satisfies the   rigidity estimate, see Definition \ref{eq:weakrig}, we can choose $A \in \Gamma$ such that
\begin{equation}\label{eq:choicephi}
\fint_{\frac 1 2 B} |\D (u- \varphi)|^n \d x = \fint_{\frac 1 2 B} |\D u - A|^n \d x\leq  C_\Gamma \fint_{B} \dist(\D u, \Gamma)^n \d x.
\end{equation}
Thus the desired estimate \eqref{eq:revholderdist} follows from Lemma \ref{lemma:multbounds}, since $N(f,\Omega')<\infty$.

\section{A quantitative unique continuation principle}\label{sec:quantita}

In this section we explain how a refinement of the proof in the previous section in fact leads to the stronger result in Theorem \ref{thm:qttive}.

\subsection{A uniform bound on the multiplicity}

The main result of this subsection is Proposition \ref{prop:unimultbounds}, which gives a quantitative improvement over Lemma \ref{lemma:multbounds} above. In order to state it,  let us introduce some notation. For any open set $\Omega' \Subset \Omega$,  we consider
for $M>0$ the class of maps
\begin{align*}
X_M(\Omega',\Omega) \equiv &\{u \in W^{1,n}(\Omega,\R^n): u \text{ satisfies \eqref{eq:incEgamma} and \eqref{eq:bdddistance}}
\}.
\end{align*}
We will use the short-hand notation $X_M$ for $X_M(\Omega',\Omega)$.
\begin{proposition}[Uniform multiplicity bounds]\label{prop:unimultbounds}
Let $\Omega \subset \R^n$ be open and let $\Omega' \Subset \Omega$. Let $\Gamma\subset \R^{n\times n}$ satisfy a   rigidity estimate. Then there is $C=C(n,K,\Gamma,M,\Omega')>0$ such that $$N(u-\varphi,\Omega')\leq C$$ for all $u \in X_M$ and all linear maps $\varphi(x) \equiv Ax$ for $A \in \Gamma$.
\end{proposition}

To prove this result,  we begin with the following simple lemma:

\begin{lemma}\label{lemma:qrconst}
Let $(f_j) \subset W^{1,n}_\loc(\Omega,\R^n)$ be a sequence of $K$-quasiregular maps and let $c \in \R^n$ be a constant. Then 
$$f_j \w c \text{ in } W^{1,n}_\loc(\Omega,\R^n) \quad \implies \quad
f_j \to c \text{ in } W^{1,n}_\loc(\Omega,\R^n) 
$$
\end{lemma}

\begin{proof}
The claim follows from \cite[Corollary 1.2]{Muller1990} and the distortion inequality \eqref{eq:qr}.
\end{proof}

We will also use the following topological result, which asserts that for discrete, open maps the supremum in \eqref{eq:defmult} can be replaced with the essential supremum.

\begin{lemma}\label{lemma:openmap}
Let $f\colon \Omega\to \R^n$ be continuous, open and discrete with $\|N(f,\Omega,\cdot)\|_{L^\infty(f(\Omega))}<\infty$. Then
$N(f,\Omega) = \|N(f,\Omega,\cdot)\|_{L^\infty(f(\Omega))}.$
\end{lemma}

\begin{proof}
Let us write $N\equiv \|N(f,\Omega,\cdot)\|_{L^\infty(f(\Omega))}\leq N(f,\Omega)$. Assume by contradiction that $N(f,\Omega,y) \ge  N+1$ for some $y \in f(\Omega)$, so let $\{x_1,\dots, x_{N+1}\}$ be pre-images of $y$ through $f$. We can find $r > 0$ sufficiently small such that $B_r(x_i) \cap B_r(x_j) = \emptyset$ for all $i \neq j$. Since $f_j$ is an open mapping, the set
\[
U\equiv \bigcap_{i = 1}^{N +1}f_j(B_r(x_i))
\]
is itself open. Thus, for a.e.\ $\tilde y$ in $U$ we have $N(f,\Omega,\tilde y)\leq N$. But this is a contradiction, since each point in $U$ has at least $N + 1$ pre-images by construction.
\end{proof}

Finally, we show the following lemma, which represents a counterpart to Lemma \ref{lemma:qrconst} in the case the limit map is non-constant.

\begin{lemma}\label{lemma:qrnonconst}
Let $(f_j) \subset C^0(\Omega,\R^n)$ be a sequence of maps converging locally uniformly to a discrete and open map $f \in C^0(\Omega,\R^n)$. Then, for any set $\Omega' \Subset \Omega$,
$$\limsup_{j \to \infty }N(f_j,\overline{\Omega'}) < + \infty.$$
\end{lemma}
\begin{proof}
We let $d(g,U,p)$ be the Brouwer degree of a continuous function $g\colon U \to \R^n$ with respect to the point $p \notin g(\partial U)$. We refer the reader to \cite[Section 2]{Fonseca1995b} for the definition. We employ \cite[Lemma 2.9, Corollary 2.10]{Martio1969} to find a system of normal neighborhoods for $f$ in $\Omega$, i.e.\ for all $x \in \Omega$, there exists $r_0 = r_0(x) > 0$ and open sets $\{U(x,r)\}_{r \le r_0}$ with the following properties
\begin{enumerate}
\item\label{prop:ze} $\lim_{r\to 0}\diam(U(x,r)) = 0$;
\item\label{prop:fi} for all $r \le r_0$, $f(\partial U(x,r)) = \partial (f(U(x,r)))$;
\item\label{prop:se} for all $0 < r < s < r_0$, $\overline{U(x,r)} \subset U(x,s)$.
\end{enumerate}
These neighborhoods actually enjoy many more properties, but these are sufficient for our purposes. As $\overline{\Omega'}$ is a compact set inside $\Omega$, we can find finitely many points $x_1,\dots, x_\ell \in \overline{\Omega'}$ and positive numbers $r_1,\dots,r_\ell$ such that 
\begin{equation}\label{eq:nn}
\overline {\Omega'} \subset \bigcup_{i= 1}^\ell U\left(x_i,\frac{r_i}{2}\right) \subset  \bigcup_{i= 1}^\ell U\left(x_i,r_i\right)\Subset \Omega.
\end{equation}
Let $B_i \equiv U\left(x_i,\frac{r_i}{2}\right)$ and $A_i \equiv U\left(x_i,r_i\right)$, for all $i =1,\dots, \ell$. Due to \ref{prop:fi}-\ref{prop:se},
\[
f(\overline{B_i}) \cap f(\partial A_i) = \emptyset, \quad \forall i.
\]
By the uniform convergence $f_j\to f$ on $\overline A_i$, we thus find $\varepsilon_i > 0$ such that for all $y \in f(\overline{B_i})\cup f_j(\overline{B_i})$ and all $j \in \N$ sufficiently large (depending on $i$),
\begin{equation}\label{eq:epsmin}
0 < \varepsilon_i \le \min\{\dist(y,f(\partial A_i)),\dist(y,f_j(\partial A_i))\}.
\end{equation}
If $j$ is sufficiently large, \cite[Proposition 2.3(i)]{Fonseca1995b} tells us that then for all such $y$
\begin{equation}\label{eq:eqdeg}
d(f_j,A_i,y) = d(f,A_i,y).
\end{equation}
Now \cite[Lemma 2.4(i)-(ii)]{Fonseca1995b} imply that
\begin{equation}\label{eq:eqN}
d(f_j,A_i,y)  = N(f_j,A_i,y) \text{ and } d(f,A_i,y) = N(f,A_i,y)\quad  \text{for a.e.\ } y \text{  in }\R^n.
\end{equation}
Combining \eqref{eq:eqdeg} and \eqref{eq:eqN}, we find that, for all $j\geq j_i$, $j_i$ sufficiently large, and for almost every $y \in  f(\overline{B_i}) \cup  f_j(\overline{B_i})$,
\[
N(f_j,A_i,y) = N(f,A_i,y).
\]
Due to Lemmas \ref{lemma:multbounds} and \ref{lemma:openmap}, we then find that for all $y \in f(\overline{B_i}) \cup  f_j(\overline{B_i})$,
\begin{equation}\label{eq:fj}
\sup_{j\geq j_i} N(f_j,\overline{B_i},y) \leq \sup_{j\geq j_i} N(f_j,A_i,y) \le N(f,A_i).
\end{equation}
Therefore, choosing $j_0 = \max\{j_1,\dots, j_\ell\}$, we can write for all $y \in f_j(\overline{\Omega'})$ and $j\geq j_0$:
\[
N(f_j,\overline \Omega',y) \le \sum_{i = 1}^\ell N(f_j,B_i,y) \overset{\eqref{eq:fj}}{\le} \ell\max_{i = 1}^\ell N(f,A_i)<\infty,
\]
by Lemma \ref{lemma:multbounds}. This concludes the proof. 
\end{proof}

Notice that the fact that Lemmas \ref{lemma:openmap} and \ref{lemma:qrnonconst} apply in our setting is due to Theorem \ref{thm:reshet}. We can finally give the proof of Proposition \ref{prop:unimultbounds}.

\begin{proof}[Proof of Proposition \ref{prop:unimultbounds}]
We argue by contradiction: let us assume that there is a sequence of solutions $u_j \in X_M$ and of matrices $A_j \in \Gamma$ such that $f_j \equiv u_j - \varphi_j$ satisfies $N(f_j,\Omega')\geq j$, where $\varphi_j(x) \equiv A_j x$ are the induced linear maps. We now split the proof into two cases.

\medskip
\textbf{Case 1:} We will not relabel subsequences. Assume there exists a subsequence of $(A_j)_j$ which is equibounded.
Then there exists a subsequence of $(f_j)_j$ which is equibounded in $W^{1,n}(\Omega,\R^n)$, since $u_j \in X_M$ for all $j$, and so $f_j \w f$ in $W^{1,n}(\Omega,\R^n)$ up to a subsequence. If $f$ is constant, then we see that the convergence is actually strong in $W^{1,n}(\Omega',\R^n)$ by Lemma \ref{lemma:qrconst}. But then
\begin{equation}\label{eq:contra2}
M^{-1}\le \int_{\Omega'}\dist(\D u_j,\Gamma)^n \d x \le  \int_{\Omega'} |\D u_j - A_j|^n \d x =  \int_{\Omega'} |\D f_j|^n \d x \to 0,
\end{equation}
a contradiction. Therefore, the limit map $f$ is non-constant. Furthermore, the weak convergence of the Jacobians invoked in the proof of Lemma \ref{lemma:qrconst} implies that $f$ is $K$-quasiregular as well (but see also \cite[Theorem 8.10.1]{Iwaniec2001}).  Moreover, due to Theorem \ref{thm:reshet}, we see that $f_j$ and $f$ are continuous for all $j$ and that $(f_j)_j$ converges in the $C^0$ topology on every compact set towards $f$.  Theorem \ref{thm:reshet} tells us that the assumptions of Lemma \ref{lemma:qrnonconst} are fulfilled, and we immediately reach a contradiction.



\medskip
\textbf{Case 2:} assume that $\liminf_{j\to \infty} |A_j| \to \infty$. We introduce the normalized maps
\[
g_j \equiv \frac{f_j}{|A_j|}.
\]
Up to non-relabeled subsequences, $(g_j)_j$ is a sequence of $K$-quasiregular mappings strongly converging in $W^{1,n}(\Omega,\R^n)$ to a linear map $g(x) = Bx$ for some $|B| = 1$. As above, we find that $g$ is quasiregular, since it is a limit of quasiregular mappings, which simply means that $\det(B) \ge K|B|^n = K$. In particular, $g$ is a homeomorphism, and Lemma \ref{lemma:qrnonconst} shows that $N(g_j,\Omega')$ is equibounded in $j$. However,
\[
N(g_j,\Omega') = N(|A_j|g_j,\Omega') = N(f_j,\Omega'),
\]
and we find a contradiction with $N(f_j,\Omega') \ge j$ for all $j$ in this case as well.
\end{proof}

\subsection{Proof of Theorem \ref{thm:qttive}}

Let $\delta > 0$. We can assume without loss of generality that $\Omega' \equiv \{x \in \Omega: \dist(x,\partial \Omega) > \delta\}$ and $U \equiv\{x \in \Omega: \dist(x,\partial \Omega) > \delta/2\}$, where we also assume $\delta$ to be sufficiently small so that $\Omega'$ and $U$ are nonempty. 
With this choice, clearly $\Omega' \Subset U \Subset \Omega$. We also let $w \equiv \dist(Du,\Gamma)^n$. The proof of Theorem \ref{thm:main} and in particular \eqref{eq:revholderdist} shows the following: there exists $\eps, C$ such that
\begin{equation}\label{eq:Ceps}
\left(\fint_B w^{1+\eps} \d x\right)^{\frac{1}{1+\eps}}\leq C \fint_B w\d x, \quad \text{for all balls $B$ with $2B\subset U$}.
\end{equation}
In particular, $\eps$ does not depend on $\delta$, but $C$ does. Indeed, as can be seen in \eqref{eq:prefinale}, $C$ depends on $N(u-\varphi,U)$, where $\varphi(x)= Ax$ for some $A \in \Gamma$. However, Proposition \ref{prop:unimultbounds} (employed with $U$ instead of $\Omega'$) tells us that $N(u-\varphi,U)$ can be bounded by a constant only depending on $n,K,\Gamma,M$ and $U$, which in turn only depends on $\delta$. Thus, Theorem \ref{thm:Ap} (employed with $U$ instead of $\Omega)$ implies the existence of $p > 1$ and $c > 0$, both depending on $n,K,\Gamma,M$ and $\delta$ such that
\begin{equation}\label{eq:ap}
\fint_{B} w^{1-p}(x)\d x \le c \left(\fint_{B}w(x)\d x\right)^{1-p},
\end{equation}
for all balls $B$ such that $32B \subset U$. We define $r_0 \equiv \frac{\delta}{128}$. Notice that, with this choice, for any $x \in \overline{\Omega'}, B_{32r_{0}}(x) \subset U$.  If we manage to show that there exists a constant $$c = c(n,K,\Gamma,M, K) > 0$$ such that for every $u \in X_M$ and for every $B = B_{r_0}(x), x \in \overline{\Omega'}$,  we have
\begin{equation}\label{eq:frombelow}
\int_B \dist(\D u, \Gamma)^n \d x \ge c > 0,
\end{equation}
then a simple covering argument that uses the compactness of $\overline {\Omega'}$ allows us to conclude the bound \eqref{eq:unif} through \eqref{eq:ap}. 

The argument to show \eqref{eq:frombelow} is again by contradiction. Assume there exist $u_j \in X_M$ and $x_j \in \overline{\Omega'}$ such that
\[
\int_{B_{r_0}(x_j)} \dist(\D u_j, \Gamma)^n \d x \le j^{-1}.
\]
Since $\Gamma$ satisfies a   rigidity estimate, we find matrices $A_j \in \Gamma$ such that
\begin{equation}\label{eq:ncontr}
\int_{B_{r_0/2}(x_j)} |\D u_j - A_j|^n\d x \le C_\Gamma \int_{B_{r_0}(x_j)} \dist(\D u_j, \Gamma)^n \d x \to 0.
\end{equation}
Set $f_j \equiv u_j - \varphi_j$, where as before $\varphi_j(x) \equiv A_jx$. We can assume that $x_j \to x_0$. Moreover, we have that $(A_j)_j$ is equibounded. Indeed, \eqref{eq:ncontr} implies
\begin{align*}
c_n r_0^n |A_j|^n & \leq \int_{B_{r_0/2}(x_j)} |A_j|^n\d x \\
& \le C\left[ \int_{B_{r_0}(x_j)} \dist(\D u_j, \Gamma)^n \d x + \int_\Omega|\D u_j|^n\d x\right]\leq C(1+M),
\end{align*}
where  $C = C(\Gamma,n,\Omega) > 0$. 
Thus, up to non-relabeled subsequences, we can also assume that $f_j \w f$ in $W^{1,n}_\loc(\Omega,\R^n)$. We will now show that $f$ is a constant, which then yields a contradiction exactly as in \eqref{eq:contra2}. To show that $f$ is constant, we simply use \eqref{eq:ncontr} to deduce that
\[
\int_{B_{r_0/4}(x_0)} |\D f|^n\d x \le \liminf_{j\to \infty} \int_{B_{r_0/4}(x_0)} |\D f_j|^n\d x \le \liminf_{j\to \infty} \int_{B_{r_0/2}(x_j)} |\D f_j|^n\d x \overset{\eqref{eq:ncontr}}{=} 0.
\]
Therefore, $f$ is constant on an open subset of $\Omega$. If it were non-constant on $\Omega$, then $f$ would be open by Theorem \ref{thm:reshet},  which is clearly impossible. This concludes the proof.

\section{Application to the Nonlinear Beltrami Equation}
\label{sec:belt}

The purpose of this short section is to prove Corollary \ref{cor:nonlinbeltr}, which we restate here for the reader's convenience:

\begin{theorem}\label{teo:nonlinbelt}
Let $\Gamma\subset \tp{CO}(2)$ be a $K$-elliptic curve, and let $u \in  W_{\loc}^{1,2}(\Omega,\C)$ be a solution to 
\begin{equation}
\label{eq:b}
\p_{\bar z} u = \mu \dist(\p_z u,\Gamma), \qquad \|\mu\|_{L^\infty(\Omega)}\leq \frac{K-1}{K+1}.
\end{equation}
 Then either $u$ is affine or $\dist(\partial_z u, \Gamma) \neq 0$ a.e.\ in $\Omega$.
\end{theorem}

As we will see,  Theorem \ref{teo:nonlinbelt} is a simple consequence of Theorem \ref{thm:main}; we also note that Theorem \ref{thm:qttive} provides an obvious quantitative version of Theorem \ref{teo:nonlinbelt}, which we will not state.

We start by recalling the  definition of the Wirtinger derivatives.  For $f\in W^{1,1}(\Omega,\R^2)$, which we write as $f=(f_1,f_2)$ in components,  we set
\[
\partial_zf \equiv \frac{1}{2}[(\partial_1f_1 + \partial_2f_2) + i(\partial_1f_2 - \partial_2f_1)]\; \text{ and }\; \partial_{\bar{z}}f \equiv \frac{1}{2}[(\partial_1f_1 - \partial_2f_2) + i(\partial_1f_2 + \partial_2f_1)].
\]
These derivatives allow us to identify $\D f \in \R^{2 \times 2}$ with a pair $(\p_z f, \p_{\bar z} f)\in \C^2$.  We note that this identification satisfies the rules
\begin{equation}
\label{eq:confcoords}
|\D f| = |\p_z f|+|\p_{\bar z} f|, \qquad \det \D f = |\p_z f|^2 - |\p_{\bar z} f|^2.
\end{equation}

Theorem \ref{teo:nonlinbelt} is an immediate consequence of Theorem \ref{thm:main} and the following lemma:

\begin{lemma}\label{lem:tec}
Let $\Gamma\subset \tp{CO}(2)$ be a $K$-elliptic curve.  A map  $u \in  W_\loc^{1,2}(\Omega,\R^2)$ is a solution to \eqref{eq:b} if and only if $\D u \in \mc E_\Gamma$ a.e.\ in $\Omega$.
\end{lemma}

\begin{proof}
Let $\Gamma=\{\gamma(t):t\in [0,1]\}$ and set $u_t(z)\equiv u(z)-\gamma(t)z$.  Using \eqref{eq:confcoords} and Definition \eqref{eq:Kqc},  we see that for any $z\in \Omega$ we have 
$$\D u (z)\in \mc E_\Gamma \quad \iff \quad |\p_{\bar z}(u_t(z))| \leq \frac{K-1}{K+1}  |\p_{z}(u_t(z))|\quad \forall t\in [0,1].$$
Since $\Gamma\subset \tp{CO}(2)$, we have $\p_{\bar z}u_t = \p_{\bar z} u$,  and hence taking the infimum over $t$ we see that 
\begin{equation}
\label{eq:envelopecalc}
\D u (z)\in \mc E_\Gamma \quad \iff \quad |\p_{\bar z}u(z)| \leq \frac{K-1}{K+1}  \dist(\p_z u(z), \Gamma).
\end{equation}
Hence it is clear that any solution to \eqref{eq:b} is a solution to the differential inclusion \eqref{eq:incEgamma}. Conversely, given a solution to the differential inclusion, we obtain a solution to \eqref{eq:b} by setting $\mu(z) \equiv \frac{\p_{\bar z}u(z)}{\dist(\p_z u(z), \Gamma)}$ if the denominator is non-zero and $\mu(z) = 0$ otherwise, which satisfies the required bounds.
\end{proof}

\begin{proof}[Proof of Theorem \ref{teo:nonlinbelt}]
By \cite{Lamy2023}, $\Gamma$ satisfies a rigidity estimate. By the lemma,  \eqref{eq:b} can be rewritten as the differential inclusion \eqref{eq:incEgamma}, and the conclusion follows from Theorem \ref{thm:main}.
\end{proof}

Note that the proof of Corollary \ref{cor:jaask} is exactly the same: in that case,  we have $\Gamma=\tp{span}(\Id)$ and it is straightforward to check that 
\[
\dist(\D u(z),\Gamma) = |\Im(\partial_{ z}u(z))| + |\partial_{\bar z}u(z)| = (1-|\mu(z)|)\left|\Im(\partial_{z}u(z))\right|.
\]
In this setting, in fact, one can avoid using the results from \cite{Lamy2023}, since the needed rigidity estimate is linear and is therefore a corollary of the classical Korn inequality, as used for instance in \cite[Step 3, Proposition 3.1]{Friesecke2002}.

\section{Application to the Monge--Amp\`ere equation}
\label{sec:MA}

Let $O \subset \R^n$ be a convex set with $B_1(0) \subset O \subset B_n(0)$, and let $\varphi\colon O \to \R$ be the unique convex function solving 
\begin{equation}\label{MAb}
\begin{cases}
\lambda^{-1}\d x \le \mu_\varphi \le \lambda \d x,\\
\varphi = 0 \text{ on }\partial O,
\end{cases}
\end{equation}
where $\lambda \ge 1$ and $\mu_\varphi$ is the Monge--Amp\`ere measure associated to $\varphi$. 
For a given $t>0$, we consider the corresponding section
\[
Z_{t} \equiv \left\{x \in O: \varphi(x)\le -t^{-1} \|\varphi\|_{L^\infty(O)}\right\}.
\]
The purpose of this section is to prove the following result:

\begin{theorem}\label{thm:MA}
Let $n = 2$, let $O$ be an open convex set satisfying $B_1\subset O \subset B_2$ and let $\varphi$ be the convex solution to \eqref{MAb}. 
There are constants $C,\varepsilon > 0$,  depending only on $\Lambda$, such that
\begin{equation}\label{eq:1eps}
\|\D^2\varphi\|_{L^{1+\varepsilon}(Z_{2})} \le C.
\end{equation}
\end{theorem}
We note that  Theorem \ref{thm:MA}, combined with standard covering arguments,  yields Corollary \ref{cor:MA}, just as  in \cite[Theorem 1.1]{DePhilippis2013a}.

\subsection{Setup}

We start by remarking that, to prove Theorem \ref{thm:MA} and Corollary \ref{cor:MA},  it is sufficient to prove an a priori $W^{2,1+\e}$-estimate. Indeed, if we write $\mu_\varphi = g\d x$ with $\lambda^{-1}\le g \le \lambda$ a.e.\ in $O$, we can mollify $g$ and solve the associated Monge--Amp\`ere problem. The unique solution we get at every step of the mollification is strictly convex by \cite[Theorem 2.19]{Figalli2017} and hence smooth due \cite[Corollary 4.43]{Figalli2017}. Every smooth solution is also bounded a priori in $L^\infty(O)$ by \cite[Theorem 2.8]{Figalli2017}. Due to the uniqueness of solutions to \eqref{MAb}, see \cite[Corollary 2.11]{Figalli2017}, the sequence of smooth solutions obtained in this way converges locally uniformly to the unique $\varphi$ solving $\mu_\varphi = g \d x$, which then inherits the apriori estimates. 

By the previous paragraph, we will now take  $g \in C^\infty(O,[\lambda^{-1},\lambda])$ and assume that $\varphi$ is the smooth, strictly convex solution to 
\begin{equation}\label{MAbeq}
\begin{cases}
\det \D^2 \varphi = g & \text{in } O,\\
\varphi = 0 & \text{on }\partial O.
\end{cases}
\end{equation}
As described in the introduction, set
\begin{equation}\label{eq:phi1}
\Phi_1(w) \equiv  \frac{\D\varphi(w) - w}{\sqrt{2}}\quad  \text{ and } \quad \Phi_2(w) \equiv \frac{\D\varphi(w) + w}{\sqrt{2}}.
\end{equation}
The convexity of $\varphi$ shows that $\Phi_2$ is a homeomorphism of $O$ onto $\Omega\equiv \Phi_2(O)$ so we can define, as in \cite[(13)]{Ambrosio2011},
\begin{equation}\label{eq:udef}
u(z) \equiv \Phi_1\circ\Phi_2^{-1}(z), \quad \forall z \in \Omega.
\end{equation}
This is the so-called Minty's correspondence between monotone and $1$-Lipschitz maps.  For solutions of \eqref{MAbeq} one can check the following, cf.\ \cite[Proposition 4.2]{Ambrosio2011}:
\begin{lemma}
\label{lemma:diffincMA}
Let $\mc A$ be the set of \textit{admissible gradients}
$$\mc A\equiv \Big\{A\in \Sym(2): |A|\leq 1,  |\tr(A)|\leq \frac{\lambda-1}{\lambda+1} (1+ \det A)\Big\}.$$
If $\varphi$ solves \eqref{MAbeq} and $u$ is as in \eqref{eq:udef}, then
\begin{equation}
\label{eq:diffincMA}
\D u(z)\in \mc A\quad  \text{  for all } z\in \Omega.
\end{equation}
In particular, $u$ is 1-Lipschitz.
\end{lemma}

The proof of \cite[Proposition 4.2]{Ambrosio2011} is complicated by the fact that the maps there were  not assumed to be smooth.  In the smooth setting that we consider here,  Lemma \ref{lemma:diffincMA} is a straightforward calculation: indeed, it is easy to see that if $\D^2 \varphi(w)$ has eigenvalues $(\lambda_i(w))_{i=1,2}$, then $\D u(z)$ is a symmetric matrix with eigenvalues $\bigl(\frac{\lambda_i(z)-1}{\lambda_i(z)+1}\bigr)_{i=1,2}$, where $z=\Phi_2^{-1}(w)$. The conclusion then follows directly from \eqref{MAbeq}. We refer the reader to Lemma \ref{lem:linked} below for a similar calculation.

In the rest of this section, we write
\[
A_0 \equiv \begin{pmatrix}
1 & 0 \\ 0 & - 1
\end{pmatrix}
\]
for the conjugation operator, and we consider the set of \textit{singular gradients}
$$\mc{S} \equiv \{M \in \tp O(2): \det(M) < 0\}=
\{RA_0R^{-1}: R\in \tp{SO}(2) \}.$$
As discussed in the introduction (for $v=\bar u$ instead of $u$), our main goal is to estimate the decay of $\dist(\D u,\mc S)$.



The next lemma is a sharper version of the results of \cite{Ambrosio2011}, in particular of \cite[Lemma 7.2]{Ambrosio2011}. The main difference to their work is that we use the operator norm instead of the Euclidean norm in the calculations.

\begin{lemma}\label{lemma:qcenvelope}
With $K=\lambda$, the $K$-quasiconformal envelope of $A_0 \mc S = \SO(2)$ satisfies
$$\mc A \subset A_0 \mc E_{\SO(2)}\cap \{M:|M|\leq 1\}.$$
\end{lemma}

\begin{proof}
For simplicity, set $k\equiv \frac{K-1}{K+1}$. Let us first compute explicitly $\mc E_{\SO(2)}$: by \eqref{eq:envelopecalc}, we have
$$\mc E_{\SO(2)}=\left\{A\in \R^{2\times 2}: |a_-|\leq k \left|1-|a_+|\right|\right\},$$
where we identify $A\in \R^{2\times 2}$ with $(a_+,a_-)\in \C^2$ as described in Section \ref{sec:belt}, i.e.\ through $A(z) = a_+ z + a_- \bar z$.   
We note the identities 
\begin{equation}
\label{eq:confidentities}
\tr(A) = 2 \Re a_+, \quad \det A = |a_+|^2- |a_-|^2, \quad |A|= |a_+|+|a_-|, 
\end{equation}
together with the fact that $A\in \Sym(2)$ if and only if $a_+\in \R$.  In particular,
$$A_0 \mc E_{\SO(2)}\cap \{M:|M|\leq 1\} = \left\{A\in \R^{2\times 2}: |a_+|\leq k (1-|a_-|)\right\}.$$
From \eqref{eq:confidentities}, with $\lambda=K$, we see that any $A\in \mc A$ satisfies
\begin{equation}
\label{eq:condA}
 |a_+|  \leq k(1-|a_-|) \frac{1+|a_-|}{2- k |a_+|}, \qquad |a_+|+|a_-|\leq 1.
\end{equation}
Due to the second condition in \eqref{eq:condA} and the fact that $k<1$ we see that $|a_-|+k |a_+|\leq 1$ holds, which in turn implies
$$\frac{1+|a_-|}{2- k |a_+|}\leq 1.$$
This shows that any $A\in \mc A$ satisfies $|a_+|\leq k (1-|a_-|)$, and the conclusion follows.
\end{proof}

\subsection{Preliminary results on sections}

Sections of convex functions play a key role in the regularity theory for the Monge--Amp\`ere equation.  In Theorem \ref{thm:MA} we assume that the convex domain $O$ is normalized, in the sense that $B_1\subset O\subset B_2$,  and so there are universal estimates on the behavior of the solution. In particular we have the following result, which the reader can find in \cite[Corollary 4.5]{Figalli2017}: 

\begin{proposition}\label{prop:sections}
In the setting of Theorem \ref{thm:MA}, there is $\delta = \delta(\lambda) > 0$ with $\dist(Z_2,\partial Z_4) \ge \delta .$
\end{proposition}

As an immediate consequence of Proposition \ref{prop:sections} and convexity we obtain the following result, which we state here as a lemma.

\begin{lemma}\label{lemma:firstclaim}
There is $c_1 = c_1(\lambda) > 0$ such that $\dist(\Phi_2(Z_2),\partial \Phi_2(Z_4)) \ge c_1.$
\end{lemma}

\begin{proof}
By the Cauchy--Schwarz inequality and the convexity of $\varphi$,  for $x,y\in O$ we can write
\begin{align*}
|\Phi_2(x) - \Phi_2(y)||x-y| & \ge \left( \Phi_2(x)-\Phi_2(y),x-y\right) \\
&= \frac{1}{\sqrt 2}\left(|x-y|^2 + \left( \D\varphi(x)-\D\varphi(y),x-y\right) \right)\\
& \ge \frac{1}{\sqrt 2}|x-y|^2.
\end{align*}
Rearranging, we see that
$|\Phi_2(x) - \Phi_2(y)| \ge |x-y|/\sqrt{2}$, and the conclusion follows from Proposition \ref{prop:sections}.
\end{proof}

The next  result that we will need, and which is  due to \cite{Caffarelli1991},  gives a universal estimate on the $C^{1,\a}$ norm of the solution. The reader can find a proof  in \cite[Theorem 4.20]{Figalli2017}.

\begin{theorem}[$C^{1,\a}$-regularity]\label{thm:C1a}
In the setting of Theorem \ref{thm:MA},  there are universal constants $\alpha, c > 0$,  which only depend on $\lambda$,  such that
$\|\varphi\|_{C^{1,\alpha}(Z_4)} \le c.$
\end{theorem}

In particular, recalling definition \eqref{eq:phi1}, we obtain:

\begin{corollary}\label{cor:holdser}
Let $B_\rho(x)\subset \Phi_2(Z_4)$ and set $\rho' \equiv c^{-1/\alpha}\rho^{1/\alpha}$. Then 
\begin{equation*}
B_{\rho'}(y) \subset \Phi_2^{-1}(B_{\rho}(x)), \qquad\text{where }  y=\Phi_2^{-1}(x).
\end{equation*}
\end{corollary}

The following lemma is an immediate application of Corollary \ref{cor:holdser}:

\begin{lemma}\label{lemma:secondclaim}
There is $N = N(\lambda)$ such that $\Phi_2(Z_2)$ can be covered by $N$ balls $(B_{\rho}(x_i))_{i = 1,\dots, N}$, with centers $x_i \in \Phi_2(Z_2)$ and radii $\rho = c_1/64$. 
\end{lemma}

\begin{proof}
Since $Z_2 \subset O \subset B_2$, we can find $N = N(\lambda)$ and $x_1,\dots, x_N \in \Phi_2(Z_2)$ such that
\[
Z_2 \subset \bigcup_{i = 1}^NB_{\rho'}(\Phi_2^{-1}(x_i)),
\]
where $\rho'=c^{-1/\a}\rho^{1/\a}$, as in Corollary \ref{cor:holdser}.
In turn, that corollary implies that
\[
\Phi_2(Z_2) \subset \bigcup_{i = 1}^N\Phi_2(B_{\rho'}(\Phi_2^{-1}(x_i))) \subset \bigcup_{i = 1}^NB_{\rho}(x_i),
\]
which concludes the proof.
\end{proof}

\subsection{Two elementary lemmas}

In order to prove Theorem \ref{thm:MA} we need two elementary but important lemmas, which further display the connection between the Monge--Amp\`ere equation and \eqref{eq:diffincMA}.

The first lemma will be used to control precisely the multiplicity of $u-S$ for $S\in \mc S$.

\begin{lemma}\label{lem:uShomeo}
Let $u$ be as in \eqref{eq:udef}.   Then $u-S$ is a homeomorphism for all $S \in \mc{S}$.
\end{lemma}

\begin{proof}
Fix any $S \in \mc{S}$. By the invariance of domain theorem, if suffices to show that $u-S$ is injective on $\Omega$ for all $S \in \mc{S}$. From \eqref{eq:udef} we have, for all $x \in O$:
\begin{equation}\label{defeq}
u\left(\frac{x + \D\varphi(x)}{\sqrt{2}}\right) = \frac{\D\varphi(x) - x}{\sqrt{2}},
\end{equation}
where we recall that $\varphi$ is a smooth, strictly convex solution to \eqref{MAbeq}. Rewrite \eqref{defeq} as
\begin{equation}\label{defeq1}
(u-S)\left(\frac{x + \D\varphi(x)}{\sqrt{2}}\right) = \frac{\D\varphi(x) - x}{\sqrt{2}} - S\left(\frac{x + \D\varphi(x)}{\sqrt{2}}\right).
\end{equation}
Since $x + \D\varphi(x)$ is a homeomorphism,  \eqref{defeq1} shows that $u-S$ is injective if and only if
\[
\Phi(x) \equiv \D\varphi(x) - x - S(x + \D\varphi(x))
\]
is injective.

If we first let $S = A_0$, we see that
\[
\Phi(x) = \left(\begin{array}{c} -2x_1\\ 2\partial_2\varphi(x)\end{array}\right).
\]
Thus, the injectivity of $\Phi$ follows from the strict convexity of $\varphi$. In the general case where $S$ does not have the form above, we can anyway find $M \in \tp O(2)$ such that
\[
A_0 = M^TSM.
\]
In this case we define
\[
\Psi(x) \equiv M^T\Phi(Mx), \qquad \psi(x) \equiv \varphi(Mx).
\]
Then $\psi$ solves again the Monge--Amp\`ere equation \eqref{MAbeq}, albeit with a different smooth right-hand side $g$. Furthermore,
\begin{align*}
\Psi(x) &=M^T\left[\D\varphi(Mx) -Mx - S(Mx + \D\varphi(Mx))\right] \\
&= M^T\D\varphi(Mx) -x - M^TSM(x + M^T\D\varphi(Mx))\\
& = \D\psi(x) - x - A_0(x + \D\psi(x)).
\end{align*}
Hence, by the case where $S=A_0$, we deduce that $\Psi$ is injective.  It follows that $\Phi$ is injective as well,  which concludes the proof.
\end{proof}

The next lemma links precisely the higher integrability of the solution to \eqref{MAbeq} with the decay of the distance of $u$ to $\mc S$.

\begin{lemma}\label{lem:linked}
Let $\varphi$ be the solution of \eqref{MAbeq} and let  $u$ be as in \eqref{eq:udef}.  In $\Omega$, we have
\begin{equation}\label{distgr}
|\D^2\varphi| \sim_{\lambda} \dist^{-1}(\D u,\mc{S})\circ \Phi_2
\end{equation}
\end{lemma}

\begin{proof}
We start once again from \eqref{defeq} and we take its differential: we obtain
\begin{equation}\label{XS}
\D u\left(\frac{x + \D\varphi(x)}{\sqrt{2}}\right)(\Id +\D^2\varphi(x)) = \D^2\varphi(x)-\Id.
\end{equation}
We set 
\[
X \equiv \D u\left(\frac{x + \D \varphi(x)}{\sqrt{2}}\right) \quad \text{ and }\quad  Y \equiv \D^2\varphi(x).
\]
By \eqref{MAbeq} we have
\begin{equation}\label{Ssym}
Y \in \Sym^+(2), \quad \text{ with } \lambda^{-1} \le \det(Y) \le \lambda.
\end{equation}
Furthermore, $\eqref{XS}$ can be rewritten as
\[
X = (Y-\Id)(Y + \Id)^{-1} = \frac{\det(Y)\Id + Y - \cof(Y) - \Id}{\det(Y + \Id)}.
\]
To prove the lemma,  we can assume that $Y$ is diagonal, i.e.
\[
Y = \begin{pmatrix} \lambda_1 & 0 \\ 0 & \lambda_2\end{pmatrix}, \quad \text{ with $\lambda_1 \ge \lambda_2$,}
\]
since the general case follows by a suitable conjugation by $M \in \tp O(2)$.    Notice for later use that, since $\lambda_1 \ge \lambda_2$, we have
\begin{equation}\label{mineig}
0 \le \lambda_2 \le \sqrt{\lambda_1\lambda_2} \le \sqrt{\lambda},
\end{equation}
by \eqref{Ssym}.  We can also compute
\begin{equation}\label{longm}
\det(Y)\Id + Y - \cof(Y) - \Id =\begin{pmatrix}
\lambda_1\lambda_2 + \lambda_1 - \lambda_2 - 1 & 0 \\ 0 & \lambda_1\lambda_2 + \lambda_2 - \lambda_1 - 1
\end{pmatrix},
\end{equation}
and hence $X$ is a diagonal matrix with elements $x_{11}\geq x_{22}$. It follows that 
\[
\dist(X,\mc{S}) = |X-A_0| = \left|\frac{\det(Y)\Id + Y - \cof(Y) - \Id - A_0\det(Y + \Id)}{\det(Y + \Id)}\right|
\]
and hence we only need to show that
\begin{equation}\label{fc}
\left|\frac{\det(Y)\Id + Y - \cof(Y) - \Id - A_0\det(Y + \Id)}{\det(Y + \Id)}\right| \sim_\lambda |Y|^{-1}
\end{equation}
to conclude the proof of the lemma.  Notice that, due to \eqref{Ssym} and \eqref{mineig},
\[
\det(Y + \Id) = \det(Y) + 1 + \tr(Y) = \lambda_1 \lambda_2 + 1 + \lambda_1 + \lambda_2 \sim_\lambda |Y|.
\]
Therefore, to show \eqref{fc}, it suffices to prove that
\begin{equation}\label{fc1}
|\det(Y)\Id + Y - \cof(Y) - \Id - A_0\det(Y + \Id)| \sim_\lambda 1.
\end{equation}
To prove this last claim, we return to \eqref{longm} and we compute
\begin{align*}
\det(Y)\Id + Y - \cof(Y) - \Id - A_0\det(Y + \Id) 
&= \begin{pmatrix}
-2\lambda_2 - 2& 0 \\ 0 & 2\lambda_1\lambda_2 + 2\lambda_2
\end{pmatrix}.
\end{align*}
The norm of this matrix is comparable to
\[
 \lambda_1\lambda_2+\lambda_2 + 1,
\]
which is clearly bounded from below by $1$ and is bounded from above by  \eqref{mineig}. This concludes the proof of \eqref{fc1} and hence of the present lemma.
\end{proof}

\subsection{Proof of Theorem \ref{thm:MA}}
Fix any $q \in [0,\infty)$. Due to Lemma \ref{lem:linked}, we see that
\begin{equation*}\label{eq:firstt}
\int_{Z_2}|\D^2\varphi|^{1+q}(x)\d x \lesssim_\Lambda \int_{Z_2}|\D^2\varphi|(x)\dist^{-q}(\D u,\mc{S})\left(\Phi_2(x)\right)\tp d x.
\end{equation*}
We now change variables according to $y = \Phi_2(x)$. Observe that
\[
\D\Phi_2(x) = \frac{\D^2\varphi(x) + \Id}{\sqrt{2}},
\]
whence, for all points $x \in O$, by \eqref{MAbeq},
\begin{equation}\label{eq:secondd}
\det(\D\Phi_2(x)) = \frac 1 2 \det(\D^2 \varphi(x) + \Id) \geq \frac 1 2(1+\tr(\D^2\varphi(x)))\geq \frac{|\D^2 \varphi(x)|}{2}.
\end{equation}
Therefore
\begin{equation}\label{eq:thirdd}
\int_{Z_2}|\D^2\varphi|^{1+q}(x)\d x \lesssim_\lambda \int_{\Phi_2(Z_2)}\dist^{-q}(\D u,\mc{S})(y)\d y.
\end{equation}
%
%
To complete the proof,  we can exploit Theorem \ref{thm:qttive} and \eqref{eq:thirdd}. However,  it is possible to avoid using Theorem \ref{thm:qttive}, and in particular Proposition \ref{prop:unimultbounds}, and so for the sake of simplicity we do so here.

Combining Lemmas \ref{lemma:diffincMA}, \ref{lemma:qcenvelope} and \ref{lem:uShomeo},  we see that the map $u - S$ is $\lambda$-quasiconformal for all $S \in \mc{S}$,  hence $\dist(\D u,\mathcal{S}) > 0$ a.e. in $\Omega$. This also allows us to go back to the proof of Theorem \ref{thm:main} and deduce that actually the constants do not depend on the multiplicity $N(u-\varphi,\Phi_2(Z_4))$, which is $1$ for every choice of $S$ (recall that, in that proof, $\varphi(x) = Sx$). In particular,  in this setting, we can  exploit Lemma \ref{lemma:revholder} directly thus getting, for some $\sigma = \sigma(\lambda) > 0$,
\begin{equation}\label{eq:finalissima}
\left(\fint_B \dist(\D u,\mathcal{S})^{2+\sigma} \d x\right)^{\frac{1}{2+\sigma}} \leq C(\lambda)\fint_{B} \dist(\D u, \mathcal{S}) \d x,
\end{equation}
for all balls $B \subset \Omega$ with $2B \subset \Omega$. 
Now we can use the same reasoning as in the proof of Theorem \ref{thm:qttive}, namely we can employ Theorem \ref{thm:Ap} to deduce the existence of $\varepsilon,c > 0$ depending only on $\sigma$ and $C$, and hence ultimately depending only on  $\lambda$,  such that
\begin{equation}\label{eq:invMA}
\fint_{B}\dist^{-\varepsilon}(\D u,\mathcal{S})\d x \le c\left(\fint_{B}\dist(\D u,\mathcal{S})\d x\right)^{-\varepsilon}
\end{equation}
for all balls $B \in \mc{B}_{32}(\Omega)$.  

Now we consider the balls $B_i\equiv B_{\rho}(x_i)$ given by Lemma \ref{lemma:secondclaim}, with $\rho \equiv c_1/64$.   By  Lemma \ref{lemma:firstclaim}, we see that  $32B_i \subset \Phi_2(Z_4) \subset \Omega$, hence \eqref{eq:invMA} holds for $B_i$ in place of $B$. In order to conclude the proof of \eqref{eq:1eps},  from  \eqref{eq:thirdd} we see that we only need to show that there exists a constant $c_2 = c_2(\lambda)>0$ such that
\[
\int_{B_i}\dist(\D u,\mathcal{S})\d x \ge c_2, \quad \forall i=1,\dots, N.
\]
To prove this claim,  it is enough to show that there exists a constant $c_2 = c_2(\lambda) > 0$ such that for all balls $B = B_\rho(x)$ and all $x \in \Phi_2(Z_2)$, we have
\begin{equation}\label{eq:beb}
\int_{B}\dist(\D u,\mathcal{S})\d x \ge c_2.
\end{equation}

To prove \eqref{eq:beb},  we consider once again the change of variables $y = \Phi_2(x)$ to write
\[
\int_{B}\dist(\D u,\mathcal{S})\d y = \int_{\Phi_2^{-1}(B)}\dist(\D u,\mathcal{S})(\Phi_2(x))\det(\D\Phi_2(x))\d x.
\]
By Lemma \ref{lem:linked} and \eqref{eq:secondd}, we obtain a constant $c_3 = c_3(\lambda) > 0$ such that 
\begin{align}\label{eq:phi2}
\begin{split}
 \int_{\Phi_2^{-1}(B)}\dist(\D u,\mathcal{S})(\Phi_2(x))\det(\D\Phi_2(x))\d x  & \ge  c_3\int_{\Phi_2^{-1}(B)}\frac{\det(\D\Phi_2(x))}{|\D^2\varphi|(x)}\d x\\
& \geq \frac{c_3}{2}|\Phi_2^{-1}(B)|.
 \end{split}
\end{align}
By Lemma \ref{lemma:firstclaim},  we have that $B \subset \Phi_2(Z_4)$ and hence we can apply Corollary \ref{cor:holdser} to find a lower bound on $|\Phi_2^{-1}(B)|$ in terms of $(\rho')^2= c^{-2/\a} \rho^{2/\a}$. 
Recalling from Theorem \ref{thm:C1a} and Lemma \ref{lemma:secondclaim} that $c,\rho$ and $\alpha$ only depend on $\lambda$, we see that this estimate, combined with \eqref{eq:phi2},  implies the required universal estimate \eqref{eq:beb}, which concludes the proof.

\appendix

\section{Proof of Theorem \ref{thm:Ap}}\label{app:a}

We will use freely the notation of Section \ref{sec:muc}. Moreover, we introduce the measure
\[
\mu(E) \equiv \int_E w(x)\d x.
\]
We can assume that $w$ is not identically $0$ in $\Omega$, otherwise the proof is finished. Now let $w$ satisfy \eqref{eq:ainf} with constants $\eps$ and $C$, where $t=2$ without loss of generality. By the exact same argument as in the proof of $(c) \implies (d)$ of \cite[Theorem 9.3.3]{Grafakos2009}, we find that
\begin{equation}\label{eq:AB}
\mu(A) \le C\left(\frac{|A|}{|B|}\right)^\frac{\eps}{1 + \eps}\mu(B), \quad \forall B \in \mc{B}_2(\Omega),\; \forall  A \subset B \text{ measurable.}
\end{equation}
We divide the proof into three steps.

\subsection*{Step 1: $w > 0$ a.e.\ in $\Omega$.}

This property can essentially be found in \cite[Lemma 14.5.1]{Iwaniec2001}, although we reproduce the argument here for completeness. 
Let 
\begin{align*}
Z \equiv \{x \in\Omega: w(x) = 0\},\qquad  Z' \equiv \{x\in {\Omega}: x \text{ is a density 1 point for } Z\}. 
\end{align*}
We need to show that $Z' = \emptyset$. Assume by contradiction that $Z'$ is not empty. For every $x_0 \in \Omega$ and every small $r$, we see that $B_r(x_0) \in \mc{B}_2(\Omega)$. Therefore, for every such $r$, \eqref{eq:AB} yields
\[
\mu(B_r(x_0)) = \mu(B_r(x_0)\setminus Z) \le C\left(\frac{|B_r(x_0)\setminus Z|}{|B_r(x_0)|}\right)^\frac{\eps}{1 + \eps}\mu(B_r(x_0)).
\]
However, since $x_0$ is a density $1$ point of $Z$, then
\[
\lim_{r\to 0} \left(\frac{|B_r(x_0)\setminus Z|}{|B_r(x_0)|}\right)^\frac{\eps}{1 + \eps}= 0,
\]
which implies that there actually exists $r_0 = r_0(x_0) > 0$ such that $w = 0$ on $B_{r_0}(x_0)$. This shows that the set $Z'$ is open in $\Omega$.

Consider once again $\eps$ and $C$ fulfilling \eqref{eq:ainf}. We claim that we can find $h_0 > 0$ only depending on $\eps$ and $C$ such that for all $B_r(x) \in \mc{B}_2(\Omega)$ and all $|h| \le h_0$, if $B_r(x +h) \in \mc{B}_2(\Om)$ and $w = 0$ a.e.\ on $B_r(x)$, then $w = 0$ a.e.\ on $B_r(x+ h)$. To this end, let $W \equiv B_{r}(x+h)\setminus B_r(x)$. We may again employ \eqref{eq:AB} to see that
\[
\mu(B_{r}(x+h)) = \mu (B_{r}(x+h)\setminus B_r(x)) \le C\left(\frac{|B_{r}(x+h)\setminus B_r(x)|}{|B_{r}(x+h)|}\right)^\frac{\eps}{1 + \eps}\mu(B_{r}(x+h)).
\]
It suffices to take $h_0$ such that
\[
C\left(\frac{|B_{r}(x+h)\setminus B_r(x)|}{|B_{r}(x+h)|}\right)^\frac{\eps}{1 + \eps} \le \frac{1}{2}
\]
for all $x \in \R^n$, $r >0$, $|h| \le h_0$. 

To reach a contradiction, we only need to show that $Z'$ is relatively closed in $\Omega$. To this end, let $x \in \Omega$ be the limit point of some sequence $(x_j)_j \subset Z'$ and let $h_0$ be as in the previous paragraph. Let $j$ be sufficiently large to ensure that
\[
|x_j- x| \le \frac 1 2\min\left\{\dist(x,\partial\Omega),h_0\right\}.
\]
By openness of $Z'$ there is $r_j$ be such that $w = 0$ a.e. on $B_{r_j}(x_j)$. We can assume that 
\[
r_j \le \frac{\dist(x,\partial\Omega)}{4}.
\]
Then we see that $B_{r_j}(x_j) \in \mc{B}_2(\Omega)$, and, setting $h \equiv x-x_j$, we also have $B_{r_j}(x_j + h) \in \mc{B}_2(\Om)$. Thus, by the property of the previous paragraph $w = 0$ a.e.\ on $B_{r_j}(x_j + h) = B_{r_j}(x)$, and hence $x \in Z'$. By connectedness, $Z' = \Omega$, contradicting our assumption that $w \neq 0$ on a set of positive measure on $\Omega$.

\subsection*{Step 2: doubling property of $\mu$.} 

We have the following property: there exist $\alpha,\beta \in( 0,1)$ depending only on $n$ and on the constants $\eps$ and $C$ of \eqref{eq:ainf} such that, for every $B \in \mc{B}_2(\Om)$ and for all measurable $A \subset B$,
\begin{equation}\label{eq:smalln}
\mu(A) \le \alpha \mu(B)\quad \implies \quad |A| \le \beta|B|.
\end{equation}
This is shown exactly as in the step $(d) \implies (e)$ of \cite[Theorem 9.3.3]{Grafakos2009}.

We now note that \eqref{eq:smalln} implies the doubling property of the measure $\mu \equiv w \d x$ for balls $B \in {\mc{B}}_{8}(\Om) \subset \mc{B}_2(\Om)$. Indeed, we can first choose any constant $1 < \lambda < 2$ such that:
\[
|(\lambda B) \setminus B| \le \frac{1-\beta}{2\beta}|B|, \quad \text{for every ball $B \subset \R^n$}.
\]
Notice that $\lambda$ only depends on $C$ and $\eps$ through $\beta$. But then \eqref{eq:smalln} implies that
\[
\mu(B) > \alpha \mu(\lambda B), \quad \text{for every ball $B \subset \R^n$ such that $\lambda B \in \mc{B}_2(\Om)$.}
\]
This can be reiterated in an obvious way to obtain a constant $c > 0$ depending on $n,\eps$ and $C$ such that the following doubling property of $\mu$ holds:
\begin{equation}\label{eq:doub}
 \mu(3B) \le c\mu(B), \quad \forall B \in {\mc{B}}_8(\Om).
\end{equation}

\subsection*{Step 3: Showing \eqref{eq:revholderprelim}.}

The doubling property \eqref{eq:doub} allows us to pass from balls to cubes. We use for cubes the same notation as for balls: $Q$ is a short-hand notation for $Q_r(x)$, the open cube of side $2r$ centered at $x$, and $\lambda Q$ is a short-hand notation for $Q_{\lambda r}(x)$. In particular, notice that for all $r> 0$ and $x \in \R^n$,
\begin{equation}\label{eq:cubesballs}
Q_{r/2}(x) \subset B_{r}(x) \subset Q_r(x).
\end{equation}
Set, for $t \ge 1$,
\[
\mc{Q}_t (\Om)\equiv \{Q: \text{Q is a cube with } tQ \subset \Omega\}.
\]
Using \eqref{eq:cubesballs}, from \eqref{eq:doub} we immediately deduce
\begin{equation}\label{eq:doubq}
\mu(3Q) \le c\mu(Q), \quad \forall Q \in \mc{Q}_{32}(\Om),
\end{equation}
and that, due to \eqref{eq:smalln}, there exist $\tilde\alpha,\tilde\beta \in( 0,1)$ depending only on $n,\eps$ and $C$ such that, for every $Q \in \mc{Q}_{32}(\Om)$ and for all measurable $A \subset Q$,
\begin{equation*}\label{eq:smallnq}
\mu(A) \le \tilde\alpha \mu(Q)\quad \implies \quad |A| \le \tilde\beta|Q|.
\end{equation*}
This can be rewritten in terms of $w$ as
\begin{equation}\label{eq:inverse}
\mu(A) \le \tilde \alpha\mu(Q) \quad\implies\quad \int_{A}w^{-1}(x)\d\mu(x) \le \tilde \beta \int_{Q}w^{-1}(x)\d\mu(x).
\end{equation}

We can finally conclude the proof of \eqref{eq:revholderprelim}. We follow the proof of $(e) \implies (f)$ of \cite[Theorem 9.3.3]{Grafakos2009}.  
In \cite[Corollary 9.2.4]{Grafakos2009}, it is shown that property \eqref{eq:inverse} implies a reverse H\"older inequality for $w^{-1}$ with respect to the measure $\mu$. The corollary is stated over $\R^n$ but its proof is clearly local, i.e.\ it only depends on the estimate \eqref{eq:inverse} inside the cube $Q$ and on properties of $w$ inside $Q$; see also  \cite[Theorem 9.2.2]{Grafakos2009}.  In particular it is easy to see that it works verbatim in our case as well. Thus, employing the same proofs as \cite[Theorem 9.2.2 and Corollary 9.2.4]{Grafakos2009}, we see that \eqref{eq:inverse} implies a reverse H\"older inequality for $w^{-1}$ with respect to the measure $\mu$, i.e.\ there exists $p > 1$, $c > 1$ only depending on $n, \tilde\alpha$ and $\tilde\beta$, and hence ultimately only on  $n,\eps,C$, such that
\begin{equation}\label{eq:inverseh}
\left(\frac{1}{\mu(Q)}\int_{Q} w^{-p}\d\mu(x)\right)^{\frac{1}{p}} \le\frac{c}{\mu(Q)}\int_{Q} w^{-1}\d\mu(x),
\end{equation}
for all $Q \in \mc Q_{32}(\Om)$. Rewrite \eqref{eq:inverseh} as:
\begin{equation}\label{eq:inversehh}
\int_{Q} w^{1-p}\d x \le c^p |Q|^p\mu(Q)^{1-p} = c^p |Q|^p\left(\int_{Q}w(x)\d x\right)^{1-p}.
\end{equation}
Due to \eqref{eq:cubesballs} and the monotonicity of the integrals, \eqref{eq:inversehh} implies
\begin{equation}\label{eq:inversehhh}
\int_{B} w^{1-p}\d x \le c^p 2^{np} |B|^p\left(\int_{B}w(x)\d x\right)^{1-p},
\end{equation}
for all balls $B \in \mc{B}_{32}(\Om)$, which is precisely \eqref{eq:revholderprelim}. This concludes the proof.

	\let\oldthebibliography\thebibliography
	\let\endoldthebibliography\endthebibliography
	\renewenvironment{thebibliography}[1]{
	\begin{oldthebibliography}{#1}
	\setlength{\itemsep}{0.5pt}
	\setlength{\parskip}{0.5pt}
	}
	{
	\end{oldthebibliography}
	}
	
	{\small
	\bibliographystyle{abbrv-andre}
	\bibliography{library}
	}

\end{document}